\documentclass[preprint,12pt]{elsarticle}

\usepackage{amssymb}
\usepackage{amsmath}

\numberwithin{equation}{section}
\newtheorem{theorem}{Theorem}[section]
\newtheorem{corollary}[theorem]{Corollary}

\newtheorem{lemma}[theorem]{Lemma}

\newtheorem{remark}[theorem]{Remark}
\newtheorem{example}[theorem]{Example}

\newtheorem{definition}[theorem]{Definition}

\newproof{proof}{Proof}

\journal{Linear and Multilinear Algebra}

\begin{document}

\begin{frontmatter}

\title{Reverse order law for the inverse along an element}

\author[1,2]{Huihui Zhu}
\ead{ahzhh08@sina.com}
\author[1]{Jianlong Chen\corref{cor}}
\ead{jlchen@seu.edu.cn}
\cortext[cor]{Corresponding author}

\author[2,3]{Pedro Patr\'{i}cio}
\ead{pedro@math.uminho.pt}

\address[1]{Department of Mathematics, Southeast University, Nanjing 210096, China.}
\address[2]{CMAT-Centro de Matem\'{a}tica, Universidade do Minho, Braga 4710-057, Portugal.}
\address[3]{Departamento de Matem\'{a}tica e Aplica\c{c}\~{o}es, Universidade do Minho, Braga 4710-057, Portugal.}

\begin{abstract}
In this paper, we introduce a new concept called left (right) g-MP inverse in a $*$-monoid. The relations of this type of generalized inverse with left inverse along an element are investigated. Also, the reverse order law for the inverse along an element is studied. Then, the existence criteria and formulae of the inverse of the product of triple elements along an element are investigated in a monoid. Finally, we further study left and right g-MP inverses, the inverse along an element in the context of rings.
\end{abstract}


\end{frontmatter}


\section { \bf Introduction}

There are many types of generalized inverses in mathematical literature, such as group inverses, Drazin inverses \cite{Drazin1}, Moore-Penrose inverses \cite{Penrose}, (left, right) inverses along an element \cite{Mary,Zhu and chen,Zhu and p} and so on. Many properties of these generalized inverses were considered in different settings. In particular, a large amount of work has been devoted to the study of the reverse order law for group inverses and Moore-Penrose inverses. However, few results have been presented concerning the reverse order law for the inverse along an element since it was introduced.

Throughout this paper, $S$ is a monoid. Let $*$ be an involution on $S$, that is the involution $*$ satisfies $(x^*)^* = x$ and $(xy)^* = y^*x^*$ for all $x,y\in S$. We call $S$ a $*$-monoid if there exists an involution on $S$.

In this article, we introduce a new concept called left (right) g-MP inverse in a $*$-monoid $S$. An element $a\in S$ is called left (resp., right) g-MP invertible if $Sa=Sa^2=Saa^*a$ (resp., $aS=a^2S=aa^*aS$). The relations of this type of generalized inverse with left inverse along an element will be considered in a monoid. Also, the reverse order law for the inverse along an element is studied. Then, the existence criteria and formulae of the inverse of the product of triple elements along an element are investigated. Finally, we further study left and right g-MP inverses, the inverse along an element in rings.

An element $a$ in  $S$ is called (von Neumann) regular if there exists $x\in S$ such that $a=axa$. Such $x$ is called an inner inverse of $a$, and is denoted by $a^{(1)}$. We call $a\in S$ left (resp., right) regular if $a\in Sa^2$ (resp., $a\in a^2S$). The symbol $a\{1\}$ means the set of all inner inverses of $a\in S$.

 Recall that an element $a\in S$ (with involution $*$) is Moore-Penrose invertible (see \cite{Penrose}) if there exists $x\in S$ satisfying the following equations
\begin{center}
${\rm(i)}~axa=a$~~ ${\rm (ii)}~xax=x$~~ ${\rm (iii)}~(ax)^*=ax$~~ ${\rm (iv)}~(xa)^*=xa$.
\end{center}
Any element $x$ satisfying the equations above is called a Moore-Penrose inverse of $a$. If such $x$ exists, then it is unique and is denoted by $a^\dag$. If $x$ satisfies the equations (i) and (iii), then $x$ is called a $\{1,3\}$-inverse of $a$, and is denoted by $a^{(1,3)}$. If $x$ satisfies the equations (i) and (iv), then $x$ is called a $\{1,4\}$-inverse of $a$, and is denoted by $a^{(1,4)}$. Recall that $a^\dag$ exists if and only if both $a^{(1,3)}$ and $a^{(1,4)}$ exist. In this case, $a^\dag=a^{(1,4)}aa^{(1,3)}$. If $x$ in equations (i) and (ii) satisfies $ax=xa$, then $a$ is group invertible. Moreover, the group inverse of $a$ is unique if it exists, and is denoted by $a^\#$. It is known that $a\in S$ is group invertible if and only if it is both left and right regular. We recall that an element $a\in S$ is EP if $a\in S^\# \cap S^\dag$ and $a^\#=a^\dag$. The symbols $S^\dag$ and $S^\#$ denote the sets of all Moore-Penrose invertible and group invertible elements in $S$, respectively. A well-known characterization of EP elements is that $a$ is EP if and only if $a\in S^\dag$ and $aa^\dag=a^\dag a$.

Green's preorders \cite{Green} in $S$ are defined by: (i) $a\leq_\mathcal{L}b \Leftrightarrow Sa \subseteq Sb \Leftrightarrow$  $a=xb$ for some $x\in S$; (ii) $a\leq_\mathcal{R}b \Leftrightarrow aS \subseteq bS \Leftrightarrow$ $a=by$ for some $y\in S$; (iii) $a\leq_\mathcal{H}b \Leftrightarrow a\leq_\mathcal{L}b ~~{\rm and}~~ a\leq_\mathcal{R}b$. Let $a,b,d\in S$. An element $b$ is a left (resp., right) inverse of $a$ along $d$ \cite{Zhu and chen} if $bad=d$ (resp., $dab=b$) and $b\leq_\mathcal{L}d$ (resp., $b\leq_\mathcal{R}d$). By $a_l^{\parallel d}$ and $a_r^{\parallel d}$ we denote a left and a right inverse of $a$ along $d$, respectively. We say that $a$ is invertible along $d$ \cite{Mary} if there exists $b$ in $S$ such that $bad=d=dab$ and $b\leq_\mathcal{H} d$. If such $b$ exists then it is called the inverse of $a$ along $d$. Moreover, it is unique and is denoted by $a^{\parallel d}$. It is known \cite{Zhu and chen} that $a$ is both left and right invertible along $d$ if and only if it is invertible along $d$ if and only if $d \leq_\mathcal{H} dad$.

\section{Left (right) g-MP inverses and reverse order law}

We begin this section with the following definition, whose properties and relations with left inverse along an element are one of the main study targets in this paper.

\begin{definition} \label{Def 1} Let $S$ be a $*$-monoid and let $a\in S$. We call $a$ left g-MP invertible if $Sa=Sa^2=Saa^*a$.
\end{definition}

We next give several examples of left g-MP invertible elements.

\begin{example}
{\rm (i) The unity 1 in a $*$-monoid is left g-MP invertible.

(ii) An EP element $a$ in a $*$-monoid $S$ is left g-MP invertible. Indeed, $a=a^\#a^2\in Sa^2$, i.e., $Sa=Sa^2$. Also, $a=aa^\dag a=(a^\dag)^*a^*a=(a^\dag)^*a^\dag aa^*a\in Saa^*a$. So, $Sa= Sa^2=Saa^*a$.

(iii) Let $S=M_2(\mathbb{C})$ be the monoid of $2 \times 2$ complex matrices and let the involution be the conjugate transpose. If $A = \left(\begin{smallmatrix} 1 &  0 \\ 1 & 0 \end{smallmatrix}\right)\in S$, then $A$ is left g-MP invertible since $A=I \cdot A^2=\frac{1}{2}I \cdot AA^*A$.}
\end{example}

Next, we give the definition of right g-MP inverse in a $*$-monoid.

\begin{definition} \label{Def 2} Let $S$ be a $*$-monoid and let $a\in S$. We call $a$ right g-MP invertible if $aS=a^2S=aa^*aS$.
\end{definition}

\begin{lemma} \label{left mary}{\rm \cite[Theorem 2.3]{Zhu and chen}} Let $a,d\in S$. Then

\emph{(i)} $a$ is left invertible along $d$ if and only if $d \leq_\mathcal{L}dad$.

\emph{(ii)} $a$ is right invertible along $d$ if and only if $d \leq_\mathcal{R}dad$.
\end{lemma}

The following theorem characterizes the relations between $Sa=Saa^*a$ and left inverse along an element in a $*$-monoid.

\begin{theorem} \label{left mary a} Let $S$ be a $*$-monoid and let $a\in S$. Then following conditions are equivalent{\rm :}

\emph{(i)} $Sa=Saa^*a$.

\emph{(ii)} $a^*$ is left invertible along $a$.

\end{theorem}

\begin{proof} (i) $\Rightarrow$ (ii) Suppose $Sa=Saa^*a$. Then $a\leq_\mathcal{L} aa^*a$. Hence, $a^*$ is left invertible along $a$ by Lemma \ref{left mary}.

(ii) $\Rightarrow$ (i) If $a^*$ is left invertible along $a$, it follows from Lemma \ref{left mary} that $a\leq_\mathcal{L} aa^*a$. So, $Sa=Saa^*a$. \hfill$\Box$
\end{proof}

Dually, we have the following result.

\begin{theorem} \label{right mary a} Let $S$ be a $*$-monoid and let $a\in S$. Then following conditions are equivalent{\rm :}

\emph{(i)} $aS=aa^*a S$.

\emph{(ii)} $a^*$ is right invertible along $a$.

\end{theorem}

It follows from \cite[Theorem 2.16]{Zhu and chen} that $aS=aa^*aS$ if and only if $Sa=Saa^*a$ if and only if $a$ is Moore-Penrose invertible.
Hence, we get

\begin{corollary} \label{left right MP} Let $S$ be a $*$-monoid and let $a\in S$. Then following conditions are equivalent{\rm :}

\emph{(i)} $a$ is Moore-Penrose invertible.

\emph{(ii)} $a^*$ is left invertible along $a$.

\emph{(iii)} $a^*$ is right invertible along $a$.

In this case, $(a^\dag)^*$ is a left (right) inverse of $a^*$ along $a$.
\end{corollary}

\begin{proof} For the convenience, we only prove that $(a^\dag)^*$ is a left inverse of $a^*$ along $a$. As $(a^\dag)^*a^*a=(aa^\dag)^*a=aa^\dag a=a$ and $(a^\dag)^*=(a^\dag aa^\dag)^*=(a^\dag)^* a^\dag a \leq_\mathcal{L}a$, then $(a^\dag)^*$ is a left inverse of $a^*$ along $a$. \hfill$\Box$
\end{proof}

According to \cite[Theorem 11]{Mary} and Corollary \ref{left right MP}, $a\in S^\dag$ if and only if $a$ is invertible along $a^*$ if and only if $a^*$ is invertible along $a$. Moreover, $a^{\parallel a^*}=a^\dag$ and $(a^*)^{\parallel a}=(a^\dag)^*$. We also remark that $a\in S^\#$ if and only if $a$ is invertible along $a$ if and only if 1 is invertible along $a$. In this case, $a^{\parallel a}=a^\#$ and $1^{\parallel a}=aa^\#$.

From Corollary \ref{left right MP}, we get that $a$ is left g-MP invertible if and only if it is left regular and Moore-Penrose invertible. If $a\in S$ is both left and right g-MP invertible, then $a\in S^\# \cap S^\dag$.

The relations between left g-MP inverse and the recently introduced notion called left inverse along an element will be given in Theorem \ref{left g-MP} below. Herein, we first give relations between \{1,3\}-inverses and \{1,4\}-inverses of an element in a $*$-monoid.

\begin{lemma} \label{MP representation1} Let $S$ be a $*$-monoid and let $a\in S$. If $a=xaa^*a$ for some $x\in S$, then $(xa)^*$ is both a $\{1,3\}$-inverse and a $\{1,4\}$-inverse of $a$ and $a^\dag=(xa)^*a(xa)^*$.
\end{lemma}

\begin{proof} According to \cite[Lemma 2.2]{Zhu and zhang}, we know that $(xa)^*$ is a \{1,3\}-inverse of $a$.

Note that
\begin{eqnarray*}
 (xa)^*a &=& a^*x^*a=(xaa^*a)^*x^*a=a^*aa^*(x^*)^2a \\
   &=& a^*(xaa^*a)a^*(x^*)^2a=a^*x(xaa^*a)a^*aa^*(x^*)^2a \\
   &=& a^*x^2(aa^*)^3(x^2)^*a.
\end{eqnarray*}
Hence, $[(xa)^*a]^*=(xa)^*a$, that is, $(xa)^*$ is a \{1,4\}-inverse of $a$. Hence, $a^\dag=a^{(1,4)}aa^{(1,3)}=(xa)^*a(xa)^*$. \hfill$\Box$
\end{proof}

\begin{lemma} \label{MP representation2} Let $S$ be a $*$-monoid and let $a\in S$. If $a=aa^*ay$ for some $y\in S$, then $(ay)^*$ is both a $\{1,3\}$-inverse and a $\{1,4\}$-inverse of $a$ and $a^\dag=(ay)^*a(ay)^*$.
\end{lemma}

It follows from Lemma \ref{left mary} that $Sa=Sa^2$ if and only if $a$ is left invertible along $a$. Also, Theorem \ref{left mary a} ensures that $Sa=Saa^*a$ if and only if $a^*$ is left invertible along $a$. Suppose $Sa=Sa^2=Saa^*a$. Then there exist $x,y\in S$ such that $a=xa^2=yaa^*a$. So, $a=y(xa^2)a^*a$, which means $a \leq_\mathcal{L} a^2a^*a$ and hence $aa^*$ is left invertible along $a$ from Lemma \ref{left mary}. One may guess whether the converse holds? that is, if $aa^*$ is left invertible along $a$ implies $Sa=Sa^2=Saa^*a$? Theorem \ref{left g-MP} below illustrates this fact.

\begin{theorem} \label{left g-MP} Let $S$ be a $*$-monoid and let $a\in S$. Then the following conditions are equivalent{\rm :}

\emph{(i)} $a$ is left g-MP invertible.

\emph{(ii)} $aa^*$ is left invertible along $a$.

\emph{(iii)} $baa^*a=a$ and $Sb\subseteq Sa$ for some $b\in S$.

In this case, $yxa$ and $baa^*b$ are left inverses of $aa^*$ along $a$, for all $x$ and $y$ such that $a=xa^2=yaa^*a$.
\end{theorem}

\begin{proof} (i) $\Rightarrow$ (ii) It is proved.

(ii) $\Rightarrow$ (i) Suppose that $aa^*$ is left invertible along $a$. It follows from Lemma \ref{left mary} that $a\leq_\mathcal{L}a^2a^*a$, which leads to $Sa=Saa^*a$.

Also, $a\leq_\mathcal{L}a^2a^*a$ means $a=ma^2a^*a$ for some $m\in S$. Moreover, $(ma^2)^*$ is a \{1,4\}-inverse of $a$ from Lemma \ref{MP representation1}. Therefore, $a=a(ma^2)^*a=a[(ma^2)^*a]^*=aa^*ma^2$, which yields $Sa=Sa^2$. So, $a$ is left g-MP invertible.

(i) $\Rightarrow$ (iii) Suppose that $a$ is left g-MP invertible. Then there exist $x,y\in S$ such that $a=xa^2=yaa^*a$. Hence, $a=yxa^2a^*a$. Take $b=yxa$. Then $baa^*a=a$ and $Sb\subseteq Sa$.

(iii) $\Rightarrow$ (i) From $Sb\subseteq Sa$, it follows that $b=na$ for some $n\in S$. Hence, $a=na^2a^*a$. Applying Lemma \ref{MP representation1}, we know that $(na^2)^*$ is a $\{1,4\}$-inverse of $a$. So, $a=a(na^2)^*a=a[(na^2)^*a]^*=aa^*na^2$, which implies $Sa=Sa^2$. Also, $baa^*a=a$ can conclude $Sa=Saa^*a$. Therefore, $a$ is left g-MP invertible.

We know that $yxa$ is a left inverse of $aa^*$ along $a$ for all $x$ and $y$ such that $a=xa^2=yaa^*a$. Indeed, $(yxa) aa^*a=y(xa^2)a^*a=yaa^*a=a$ and $yxa \leq_\mathcal{L} a$.

Similarly, one can check that $baa^*b$ is a left inverse of $aa^*$ along $a$.  \hfill $\Box$
\end{proof}

Dually, we obtain

\begin{theorem} \label{right g-MP} Let $S$ be a $*$-monoid and let $a\in S$. Then the following conditions are equivalent{\rm :}

\emph{(i)} $a$ is right g-MP invertible.

\emph{(ii)} $a^*a$ is right invertible along $a$.

\emph{(iii)} $aa^*ac=a$ and $cS\subseteq aS$ for some $c\in S$.

In this case, $ast$ and $ca^*ac$ are right inverses of $a^*a$ along $a$, for all $s$ and $t$ such that $a=a^2s=aa^*at$.
\end{theorem}

We next consider some relations among left and right g-MP invertibilities and EP elements.

\begin{theorem} Let $S$ be a $*$-monoid and let $a\in S$. Then the following conditions are equivalent{\rm :}

\emph{(i)} $a$ is {\rm EP}.

\emph{(ii)} $a$ is left g-MP invertible and $aS=a^*S$.

\emph{(iii)} $a$ is right g-MP invertible and $aS=a^*S$.
\end{theorem}

\begin{proof} (ii) $\Rightarrow$ (iii) Suppose that $a$ is left g-MP invertible. Then $a$ is Moore-Penrose invertible. Hence, $aS=aa^*(a^\dag)^*S \subseteq aa^*S=a^2S$ since $aS=a^*S$. Therefore, $a$ is right g-MP invertible.

(iii) $\Rightarrow$ (ii) By noting that $aS=a^*S$ implies $Sa=Sa^*$.

Note that the condition (ii) or (iii) implies $a\in S^\#$. The result follows. Indeed, from \cite[Corollary 3]{Patricio and Puystjens}, it is known that $a$ is EP if and only if $aS=a^*S$ and $a\in S^\#$ if and only if $aS=a^*S$ and $a\in S^\dag$.   \hfill$\Box$
\end{proof}

It is well known that the reverse order law holds for the classical inverse in $S$.
More precisely, $(ab)^{-1}$ = $b^{-1}a^{-1}$ for any invertible elements $a$ and $b$ in $S$.
However, $(ab)^{\parallel d}$ = $b^{\parallel d}a^{\parallel d}$ does not hold in general in $S$.
For instance, in the semigroup of 2 by 2 complex matrices, take
$a^2=a=(\begin{smallmatrix} 1 & 0 \\ 1 & 0 \end{smallmatrix})$ and $d=(\begin{smallmatrix} 1 & 1 \\ 0 & 0 \end{smallmatrix})$, then $a^{\parallel d}=(\begin{smallmatrix} \frac{1}{2} & \frac{1}{2} \\ 0 & 0 \end{smallmatrix})$. However, $(a^2)^{\parallel d}=(\begin{smallmatrix} \frac{1}{2} & \frac{1}{2} \\ 0 & 0 \end{smallmatrix}) \neq (\begin{smallmatrix} \frac{1}{4} & \frac{1}{4} \\ 0 & 0 \end{smallmatrix})=a^{\parallel d} a^{\parallel d}$.

We remark that $a$ is invertible along $d$ and $ad=da$ imply that $d$ is group invertible in $S$. Indeed, it is known that $a^{\parallel d}$ exists if and only if $d \leq_\mathcal{H} dad$. As $ad=da$, then $d \leq_\mathcal{H} dad$ implies $d \in d^2S \cap Sd^2$. Hence, $d\in S^\#$.

Next, we consider the reverse order law for the inverse along an element under commutativity condition.

\begin{lemma} {\rm \cite[Theorem 10]{Mary}} Let $a,d\in S$ with $ad=da$. If $a^{\parallel d}$ exists, then $a^{\parallel d}$ commutes with $a$ and $d$.
\end{lemma}

\begin{theorem} Let $a,b,d \in S$ with $ad=da$. If $a^{\parallel d}$ and $b^{\parallel d}$ exist, then

\emph{(i)} $(ab)^{\parallel d}$ exists and $(ab)^{\parallel d}=b^{\parallel d} a^{\parallel d}$.

\emph{(ii)} $(ba)^{\parallel d}$ exists and $(ba)^{\parallel d}=a^{\parallel d} b^{\parallel d}$.
\end{theorem}

\begin{proof} (i) Since $b^{\parallel d}$ can be written as $yd$ for some $y\in S$, we have

$b^{\parallel d}a^{\parallel d}abd=b^{\parallel d}aa^{\parallel d}bd=y(daa^{\parallel d})bd=ydbd=b^{\parallel d}bd=d$ and $dabb^{\parallel d} a^{\parallel d}=adbb^{\parallel d} a^{\parallel d}=ada^{\parallel d}=daa^{\parallel d}=d$.

As $a^{\parallel d}\leq_\mathcal{H} d$ and $b^{\parallel d}\leq_\mathcal{H} d$, then $b^{\parallel d} a^{\parallel d} \leq_\mathcal{H} d$.

Therefore, $ab$ is invertible along $d$ and $(ab)^{\parallel d}=b^{\parallel d} a^{\parallel d}$.

(ii) can be proved similarly. \hfill$\Box$
\end{proof}

We next consider the existence criteria and formulae of left inverse of the product of triple elements along an element.

\begin{theorem} \label{left mary product} Let $a,b,d\in S$. If $a$ is left invertible along $d$, then the following conditions are equivalent{\rm :}

\emph{(i)} $b$ is left invertible along $d$.

\emph{(ii)} $adb$ is left invertible along $d$.

In this case, $(adb)_l^{\parallel d}=b_l^{\parallel d} bya_l^{\parallel d}$, where $y\in S$ satisfies $d=ydbd$.
\end{theorem}

\begin{proof} (i) $\Rightarrow$ (ii) Suppose that $b$ is left invertible along $d$. Then $d\leq_\mathcal{L} dbd$ by Lemma \ref{left mary}, i.e., $d=ydbd$ for some $y\in S$. Also, as $a$ is left invertible along $d$, then $d=xdad$ for some $x\in S$. Hence, $d=y(xdad)bd$ and so $d\leq_\mathcal{L} dadbd$. Therefore, $adb$ is left invertible along $d$ from Lemma \ref{left mary}.

(ii) $\Rightarrow$ (i) As $adb$ is left invertible along $d$, then $d\leq_\mathcal{L} d(adb)d$, which implies $d\leq_\mathcal{L} dbd$. Again, Lemma \ref{left mary} guarantees that $b$ is left invertible along $d$.

We next show that $b_l^{\parallel d} bya_l^{\parallel d}$ is a left inverse of $adb$ along $d$ for all $y\in S$ such that $d=ydbd$. Indeed, we have
$(b_l^{\parallel d} bya_l^{\parallel d}) adbd=b_l^{\parallel d} bydbd=b_l^{\parallel d} bd=d$
and $b_l^{\parallel d} bya_l^{\parallel d} \leq_\mathcal{L} d$ since $a_l^{\parallel d} \leq_\mathcal{L} d$.   \hfill$\Box$
\end{proof}

Dually, it follows that

\begin{theorem} \label{right mary product} Let $a,b,d\in S$. If $a$ is right invertible along $d$, then the following conditions are equivalent{\rm :}

\emph{(i)} $b$ is right invertible along $d$.

\emph{(ii)} $bda$ is right invertible along $d$.

In this case, $(bda)_r^{\parallel d}=a_r^{\parallel d} tbb_r^{\parallel d}$, where $t\in S$ satisfies $d=dbdt$.
\end{theorem}

Next, we present two lemmas which play an important role in the proof of Theorem \ref{mary product} below, which gives equivalences among the inverses of $b$, $adb$ and $bda$ along $d$, under certain conditions. The symbol $S^{-1}$ denotes the set of all invertible elements in $S$.

\begin{lemma} \label{chen 1} Let $e^2=e\in S$ and let $a\in S$. Then $a$ is invertible along $e$ if and only if $eae \in (eSe)^{-1}$.
\end{lemma}

\begin{proof} It is known that $a$ is invertible along $e$ if and only if $Se=Seae$ and $eS=eaeS$ if and only if $eSe=eSe(eae)=(eae)eSe$ if and only if $eae\in (eSe)^{-1}$. \hfill$\Box$
\end{proof}

\begin{lemma} \label{chen 2} Let $a,d\in S$ and let $d$ be regular with an inner inverse $d^{(1)}$. Then the following conditions are equivalent{\rm :}

\emph{(i)} $a$ is invertible along $d$.

\emph{(ii)} $da$ is invertible along $dd^{(1)}$.

\emph{(iii)} $ad$ is invertible along $d^{(1)}d$.
\end{lemma}

\begin{proof} (i) $\Rightarrow$ (ii) Suppose that $a$ is invertible along $d$.  Then $dS = dadS$ and $Sd = Sdad$. As $dS=dd^{(1)}S$, then it follows that $dd^{(1)}S = dadd^{(1)}S=dd^{(1)}dadd^{(1)}S$ and $Sdd^{(1)} = Sdadd^{(1)}=Sdd^{(1)}dadd^{(1)}$. So, $da$ is invertible along $dd^{(1)}$.

(ii) $\Rightarrow$ (i) As $da$ is invertible along $dd^{(1)}$, then $dd^{(1)}S = dd^{(1)}dadd^{(1)}S=dadd^{(1)}S$ and $Sdd^{(1)} = Sdd^{(1)}da d d^{(1)}=Sdadd^{(1)}$. From $dS=dd^{(1)}S$, it follows that $dS = dadS$ and $Sd = Sdadd^{(1)}d=Sdad$, i.e., $a$ is invertible along $d$.

(i) $\Leftrightarrow$ (iii) It is similar to the proof of (i) $\Leftrightarrow$ (ii). \hfill$\Box$
\end{proof}

\begin{theorem} \label{mary product} Let $a,b,d\in S$. If $a$ is invertible along $d$, then the following conditions are equivalent{\rm :}

\emph{(i)} $b$ is invertible along $d$.

\emph{(ii)} $adb$ is invertible along $d$.

\emph{(iii)} $bda$ is invertible along $d$.

In this case, $(adb)^{\parallel d}=b^{\parallel d} d^{(1)}a^{\parallel d}$ and $(bda)^{\parallel d}=a^{\parallel d} d^{(1)}b^{\parallel d}$ for all choices $d^{(1)}\in d\{1\}$.
\end{theorem}

\begin{proof} It is known that $a$ is invertible along $d$ implies that $d$ is regular (see \cite{Mary and Patricio}). Let $d^{(1)}$ be an inner inverse of $d$ and let $e=dd^{(1)}$ and $f=d^{(1)}d$. As $a$ is invertible along $d$, then $e(da)e\in (eSe)^{-1}$ and $f(ad)f \in (fSf)^{-1}$ by Lemmas \ref{chen 1} and \ref{chen 2}.

(i) $\Leftrightarrow$ (ii) It follows from Lemmas \ref{chen 1} and \ref{chen 2} that $b$ is invertible along $d$ if and only if $e(db)e\in (eSe)^{-1}$, and $adb$ is invertible along $d$ if and only if $e(dadb)e \in (eSe)^{-1}$. Note that $e(dadb)e=e(da)e \cdot e(db)e$ and $e(da)e\in (eSe)^{-1}$. Hence, $e(dadb)e \in (eSe)^{-1}$ if and only if $e(db)e\in (eSe)^{-1}$.

(i) $\Leftrightarrow$ (iii) It is similar to (i) $\Leftrightarrow$ (ii) by noting $f(bdad)f=f(bd)f \cdot f(ad)f$.

We next show that $m=b^{\parallel d} d^{(1)}a^{\parallel d}$ is the inverse of $adb$ along $d$.

We have $a^{\parallel d}ad=d=daa^{\parallel d}$ and $a^{\parallel d}=x_1d=dx_2$ for some $x_1,x_2\in S$. Also, $b^{\parallel d}bd=d=dbb^{\parallel d}$ and $b^{\parallel d}=y_1d=dy_2$ for some $y_1,y_2\in S$.

It follows that
\begin{eqnarray*}
madbd &=&b^{\parallel d} d^{(1)}(a^{\parallel d} ad)bd=b^{\parallel d} d^{(1)}dbd\\
&=&y_1 d d^{(1)}dbd=y_1dbd=b^{\parallel d}bd\\
&=&d
\end{eqnarray*}
and

\begin{eqnarray*}
dadbm&=&da(dbb^{\parallel d}) d^{(1)}a^{\parallel d}=dadd^{(1)}a^{\parallel d}\\
&=&dadd^{(1)}dx_2=dadx_2=daa^{\parallel d}\\
&=&d.
\end{eqnarray*}

As $m=b^{\parallel d} d^{(1)}a^{\parallel d}=dy_2d^{(1)}a^{\parallel d}=b^{\parallel d} d^{(1)}x_1d$, then $m\leq_\mathcal{H}d$.

Hence, $(adb)^{\parallel d}=b^{\parallel d} d^{(1)}a^{\parallel d}$.

Similarly, we can check $(bda)^{\parallel d}=a^{\parallel d} d^{(1)}b^{\parallel d}$. \hfill$\Box$
\end{proof}

\begin{remark} \label{semigroupmark} {\rm The assumption ``$a$ is invertible along $d$'' in Theorem \ref{mary product} can not be dropped. Indeed, take $S$ be the monoid of all infinite complex matrices with finite nonzero elements in each column and let $d=1\in S$, $a=\Sigma_{i=1}^\infty e_{i,i+1}\in S$ and $b=\Sigma_{i=1}^\infty e_{i+1,i}\in S$, where $e_{i,j}$ denotes the infinite matrix whose $(i,j)$-entry is 1 and other entries are zero. Then $1=adb=ab$ is invertible along $1$. However, $\Sigma_{i=2}^\infty e_{i,i}=bda=ba$ is not invertible along 1.}
\end{remark}

As a consequence, we present a result on the reverse order law for the inverse of the product of triple elements along an element.

\begin{corollary}  Let $a,b,d\in S$. If $a$ is invertible along $d$ and $d\in S^\#$, then the following conditions are equivalent{\rm :}

\emph{(i)} $b$ is invertible along $d$.

\emph{(ii)} $adb$ is invertible along $d$.

\emph{(iii)} $bda$ is invertible along $d$.

In this case, $(adb)^{\parallel d}=b^{\parallel d} d^{\parallel d} a^{\parallel d}$ and  $(bda)^{\parallel d}=a^{\parallel d} d^{\parallel d} b^{\parallel d}$.
\end{corollary}

As special results of Theorem \ref{mary product}, it follows that

\begin{corollary} \label{coro product} Let $b,d\in S$. If $d\in S^\#$, then the following conditions are equivalent{\rm :}

\emph{(i)} $b$ is invertible along $d$.

\emph{(ii)} $bd$ is invertible along $d$.

\emph{(iii)} $db$ is invertible along $d$.

In this case, $(bd)^{\parallel d}=d^{\parallel d} b^{\parallel d}$ and $(db)^{\parallel d}=b^{\parallel d} d^{\parallel d}$.
\end{corollary}

\begin{proof} Since $1^{\parallel d}=dd^\#$, $d^{\parallel d}=d^\#$ and $b^{\parallel d}=dx$ for some $x\in S$, we have $(bd)^{\parallel d}=1^{\parallel d}d^{(1)}b^{\parallel d}=d^\#dd^{(1)}dx=d^\#dx=d^\#b^{\parallel d}=d^{\parallel d} b^{\parallel d}$.

We may use the same reasoning to obtain $(db)^{\parallel d}=b^{\parallel d} d^{\parallel d}$.  \hfill$\Box$
\end{proof}

Suppose $a=d^*$ in Theorem \ref{mary product}. Then it follows that

\begin{corollary}  \label{lastc} Let $S$ be a $*$-monoid and let $b,d\in S$. If $d\in S^\dag$, then the following conditions are equivalent{\rm :}

\emph{(i)} $b$ is invertible along $d$.

\emph{(ii)} $d^*db$ is invertible along $d$.

\emph{(iii)} $bdd^*$ is invertible along $d$.

In this case, $(d^*db)^{\parallel d}=b^{\parallel d}d^{\parallel d^*}(d^*)^{\parallel d}$ and $(bdd^*)^{\parallel d}=(d^*)^{\parallel d}d^{\parallel d^*}b^{\parallel d}$.
\end{corollary}

\begin{proof} By Theorem \ref{mary product}, we have $(d^*db)^{\parallel d}=b^{\parallel d}d^{(1)}(d^*)^{\parallel d}$. Note that $(d^*)^{\parallel d}=(d^\dag)^*$ and $b^{\parallel d}$ can be written as $xd$ for an appropriate $x$ in $S$. Hence, $(d^*db)^{\parallel d}=xdd^{(1)}dd^\dag(d^\dag)^*=xdd^\dag(d^\dag)^*=b^{\parallel d}d^{\parallel d^*}(d^*)^{\parallel d}$.

Similarly, $(bdd^*)^{\parallel d}=(d^*)^{\parallel d}d^{\parallel d^*}b^{\parallel d}$. \hfill$\Box$
\end{proof}

\section{Further results in rings}

In this section, let $R$ be an associative unital ring. An involution $*: R \to R; a \mapsto a^*$ is an anti-isomorphism on $R$ satisfying $(a^*)^* = a$, $(ab)^* = b^*a^*$ and $(a+b)^* = a^* + b^*$ for all $a,b\in R$.

We next begin with two lemmas, which play an important role in the sequel.

\begin{lemma} \label{Jlemma} Let $a,b\in R$. Then we have

\emph{(i)} If $(1+ab)x=1$ for some $x\in R$, then $(1+ba)(1-bxa)=1$.

\emph{(ii)} If $y(1+ab)=1$ for some $y\in R$, then $(1-bya)(1+ba)=1$.
\end{lemma}

By the symbols $R_l^{-1}$, $R_r^{-1}$ and $ R^{-1}$ we denote the sets of all left invertible, right invertible and invertible elements in $R$, respectively. It follows from Lemma \ref{Jlemma} that $1+ab\in R_l^{-1}$ if and only if $1+ba\in R_l^{-1}$, and $1+ab\in R_r^{-1}$ if and only if $1+ba\in R_r^{-1}$. Further, $1+ab\in R^{-1}$ if and only if $1+ba \in R^{-1}$. In this case, $(1+ba)^{-1}=1-b(1+ab)^{-1}a$, which is known as Jacobson's Lemma.

\begin{lemma} \label{left right mary} {\rm \cite[Corollaries 3.3 and 3.5]{Zhu and chen}} Let $a,m\in R$ with $m$ regular. Then the following conditions are equivalent{\rm :}

\emph{(i)} $a$ is (left, right) invertible along $m$.

\emph{(ii)} $u=ma +1- mm^{(1)}$ is (left, right) invertible.

\emph{(iii)} $v=am+1-m^{(1)}m$ is (left, right) invertible.
\end{lemma}

It is known that $1^{\parallel a}$ exists if and only if $a^\#$ exists in a ring $R$. Hence, Lemma \ref{left right mary} also gives an existence criterion of group inverse, that is, $a$ is group invertible if and only if $a+1-aa^{(1)}$ is invertible if and only if $a+1-a^{(1)}a$ is invertible, for a regular element $a\in R$.

Given an element $a\in R$, we use the symbols $a_l^{-1}$ and $a_r^{-1}$ to denote a left and a right inverse of $a$, respectively.

Applying  Corollary \ref{left right MP} and Lemma \ref{left right mary}, we derive the following result, which recovers the classical existence criterion of Moore-Penrose inverse (see, e.g. \cite[Theorem 1.2]{Patricio and Mendes}) in a ring.

\begin{theorem} \label{onemp} Let $R$ be a ring with involution and let $a\in R$ be regular. Then the following conditions are equivalent{\rm :}

\emph{(i)} $a\in R^\dag$.

\emph{(ii)} $u=aa^*+1-aa^{(1)}$ is left invertible.

\emph{(iii)} $v=a^*a+1-a^{(1)}a$ is left invertible.

\emph{(iv)} $u=aa^*+1-aa^{(1)}$ is right invertible.

\emph{(v)} $v=a^*a+1-a^{(1)}a$ is right invertible.

In this case, $a^\dag=(u_l^{-1}a)^*a(u_l^{-1}a)^*=(av_r^{-1})^*a(av_r^{-1})^*$.
\end{theorem}

\begin{proof} As $u=aa^*+1-aa^{(1)}$, then $ua=aa^*a$. It follows that $a=u_l^{-1}aa^*a$ since $u$ is left invertible.
By Lemma \ref{MP representation1}, we know that $(u_l^{-1}a)^*$ is both a $\{1,3\}$-inverse and a $\{1,4\}$-inverse of $a$.
Hence, $a^\dag=a^{(1,4)}aa^{(1,3)}=(u_l^{-1}a)^*a(u_l^{-1}a)^*$.

Similarly, as $v=a^*a+1-a^{(1)}a$ is right invertible, then $a=aa^*av_r^{-1}$. Lemma \ref{MP representation2} ensures $a^\dag=(av_r^{-1})^*a(av_r^{-1})^*$. \hfill$\Box$
\end{proof}

As a corollary of Theorem \ref{onemp}, it follows that

\begin{corollary} \label{conemp} Let $R$ be a ring with involution and let $a\in R$ be regular. Then the following conditions are equivalent{\rm :}

\emph{(i)} $a\in R^\dag$.

\emph{(ii)} $u=aa^*+1-aa^{(1)}$ is invertible.

\emph{(iii)} $v=a^*a+1-a^{(1)}a$ is invertible.

In this case, $a^\dag=(u^{-1}a)^*a(u^{-1}a)^*=(av^{-1})^*a(av^{-1})^*$.
\end{corollary}

From Theorem \ref{left g-MP} and Lemma \ref{left right mary}, it follows that

\begin{corollary} Let $R$ be a ring with involution and let $a\in R$ be regular. Then the following conditions are equivalent{\rm :}

\emph{(i)} $a$ is left g-MP invertible.

\emph{(ii)} $a^2a^*+1-aa^{(1)}$ is left invertible.

\emph{(iii)} $aa^*a+1-a^{(1)}a$ is left invertible.
\end{corollary}

By Theorem \ref{right g-MP} and Lemma \ref{left right mary}, we get

\begin{corollary} Let $R$ be a ring with involution and let $a\in R$ be regular. Then the following conditions are equivalent{\rm :}

\emph{(i)} $a$ is right g-MP invertible.

\emph{(ii)} $a^*a^2+1-a^{(1)}a$ is right invertible.

\emph{(iii)} $aa^*a+1-aa^{(1)}$ is right invertible.
\end{corollary}

The following theorem gives characterizations of left and right g-MP inverses in terms of units in a ring.

\begin{theorem} Let $R$ be a ring with involution and let $a\in R$ be regular. Then the following conditions are equivalent{\rm :}

\emph{(i)} $a$ is both left and right g-MP invertible.

\emph{(ii)} $u=aa^*a+1-aa^{(1)}$ is invertible.

\emph{(iii)} $v=aa^*a+1-a^{(1)}a$ is invertible.
\end{theorem}

\begin{proof}

(i) $\Rightarrow$ (ii) If $a$ is both left and right g-MP invertible, then $a$ is both group invertible and Moore-Penrose invertible. As $a$ is group invertible, then it follows that $a+1-aa^{(1)}\in R^{-1}$.  Also, $a$ is Moore-Penrose invertible implies that $aa^*+1-aa^{(1)}\in R^{-1}$ by Corollary \ref{conemp}. Hence, $aa^*aa^{(1)}+1-aa^{(1)}\in R^{-1}$ by Jacobson's Lemma. So, $(aa^*aa^{(1)}+1-aa^{(1)})(a+1-aa^{(1)})=u\in R^{-1}$.

(ii) $\Rightarrow$ (i) As $u=aa^*a+1-aa^{(1)} \in R^{-1}$, then $u'=a^*a^2+1-a^{(1)}a \in R^{-1}$ from Jacobson's Lemma. Hence, $au'=aa^*a^2$ and $a=aa^*a^2(u')^{-1}\in aa^*aR$, which means that $a$ is Moore-Penrose invertible by \cite[Theorem 2.16]{Zhu and chen}. Applying Corollary \ref{conemp}, we know that if $a$ is Moore-Penrose invertible then $aa^*+1-aa^{(1)} \in R^{-1}$ and hence $aa^*aa^{(1)}+1-aa^{(1)} \in R^{-1}$ according to Jacobson's Lemma. Note that $a+1-aa^{(1)}=(aa^*aa^{(1)}+1-aa^{(1)})^{-1}u\in R^{-1}$. It follows that $a$ is group invertible. Thus, $a$ is both left and right g-MP invertible.

(i) $\Leftrightarrow$ (iii) It is similar to (i) $\Leftrightarrow$ (ii). \hfill$\Box$
\end{proof}

Given a regular element $d\in R$, it is known from \cite[Theorem 3.2]{Mary and Patricio} that $a^{\parallel d}$ exists if and only if $da+1-dd^{(1)}$ is invertible if and only if $ad+1-d^{(1)}d$ is invertible. Moreover, $a^{\parallel d}=(da+1-dd^{(1)})^{-1}d=d(ad+1-d^{(1)}d)^{-1}$. Applying this, we have

\begin{corollary} \label{main mary product} Let $a,b,d\in R$ and let $a$ be invertible along $d$. Then the following conditions are equivalent{\rm :}

\emph{(i)} $b$ is invertible along $d$.

\emph{(ii)} $adb$ is invertible along $d$.

\emph{(iii)} $bda$ is invertible along $d$.

In this case, $(adb)^{\parallel d}=b^{\parallel d} v^{-1}$ and  $(bda)^{\parallel d}=u^{-1} b^{\parallel d}$, where $u=da+1-dd^{(1)}$ and $v=ad+1-d^{(1)}d$.
\end{corollary}

\begin{proof}
It follows from Theorem \ref{mary product} that $(adb)^{\parallel d}=b^{\parallel d} d^{(1)} a^{\parallel d}$. As $a^{\parallel d}=dv^{-1}$ and $b^{\parallel d}=y_1d$ for some $y\in R$, then $(adb)^{\parallel d}=y_1d d^{(1)} dv^{-1}=y_1d v^{-1}=b^{\parallel d} v^{-1}$.

Similarly, $(bda)^{\parallel d}=u^{-1} b^{\parallel d}$.  \hfill$\Box$
\end{proof}

\begin{corollary} \label{main mary product inverse} Let $a,b,d\in R$ and let $a$ be invertible along $d$. Then the following conditions are equivalent{\rm :}

\emph{(i)} $b$ is invertible along $d$.

\emph{(ii)} $u=dadb+1-dd^{(1)}$ is invertible.

In this case, $b^{\parallel d}=u^{-1}dad$.
\end{corollary}

\begin{proof} Since $u=dadb+1-dd^{(1)}=(dadd^{(1)}+1-dd^{(1)})(db+1-dd^{(1)})$, it follows that $db+1-dd^{(1)}=(dadd^{(1)}+1-dd^{(1)})^{-1}u$. Applying \cite[Theorem 3.2]{Mary and Patricio}, we have $b^{\parallel d}=(db+1-dd^{(1)})^{-1}d=u^{-1}(dadd^{(1)}+1-dd^{(1)})d=u^{-1}dad$. \hfill$\Box$
\end{proof}

\section{Remarks and questions}

We close this section with a remark and a question:\\

4.1. It is known that $ab$ is invertible along $d$ may not imply $ba$ is invertible along $d$ in a monoid (see Remark \ref{semigroupmark}). Assume that $ab$ and $ba$ are both invertible along $d$ in a monoid $S$.  Does Cline's formula for the inverse along an element hold? i.e., if $(ab)^{\parallel d}=a((ba)^{\parallel d})^2b$? In fact, this formula does not hold. For instance, let $S$ be the monoid of 2 by 2 complex matrices. Set $a^2=a=(\begin{smallmatrix} 1 & 0 \\ 1 & 0 \end{smallmatrix})\in S$ , $b=(\begin{smallmatrix} 1 & 0 \\ 0 & 1 \end{smallmatrix})\in S$ and $d=(\begin{smallmatrix} 1 & 1 \\ 0 & 0 \end{smallmatrix})\in S$. Then $a^{\parallel d}=(\begin{smallmatrix} \frac{1}{2} & \frac{1}{2} \\ 0 & 0 \end{smallmatrix})$. But $a((ba)^{\parallel d})^2b=\frac{1}{4} (\begin{smallmatrix} 1 & 1\\ 1 & 1 \end{smallmatrix}) \neq (\begin{smallmatrix} \frac{1}{2} & \frac{1}{2} \\ 0 & 0 \end{smallmatrix})= (ab)^{\parallel d}$.

4.2. Suppose that $ab$ and $ba$ are both invertible along $d$ in a monoid. Can we give a necessary and sufficient condition such that Cline's formula for the inverse along an element to hold.

\bigskip
\centerline {\bf ACKNOWLEDGMENTS}
\vskip 2mm

The authors are highly grateful to the referees for their valuable comments which led to improvements of this paper. In particular, the condition (iii) in Theorems 2.10 and 2.11, Theorems 2.15 and 2.16 were suggested to the authors by one referee.  The first author gratefully acknowledges China Scholarship Council for giving him a pursing for his joint PhD study in Department of Mathematics and Applications, University of Minho, Portugal. Also, he gratefully acknowledges CMAT-Centro de Matem\'{a}tica, Universidade do Minho for providing him a working office. This research is supported by the National Natural Science Foundation of China (No. 11371089), the Specialized Research Fund for the Doctoral Program of Higher Education (No. 20120092110020), the Natural Science Foundation of Jiangsu Province (No. BK20141327), the Scientific  Innovation Research  of College Graduates in Jiangsu Province (No. CXLX13-072), the Scientific Research Foundation of Graduate School of Southeast University, the FEDER Funds through ¡®Programa Operacional Factores de Competitividade-COMPETE', the Portuguese Funds through FCT- `Funda\c{c}\~{a}o para a Ci\^{e}ncia e a Tecnologia', within the project PEst-OE/MAT/UI0013/2014.

\begin{flushleft}
{\bf References}
\end{flushleft}


\begin{thebibliography}{00}

\bibitem{Drazin1} M.P. Drazin, Pseudo-inverses in associative rings and semigroups, Amer. Math. Monthly. 1958;65:506-514.

\bibitem{Green} J.A. Green, On the structure of semigroups, Ann. Math. 1951;54:163-172.

\bibitem{Mary} X. Mary, On generalized inverses and Green's relations, Linear Algebra Appl. 2011;434:1836-1844.

\bibitem{Mary and Patricio} X. Mary, P. Patr\'{i}cio, Generalized inverses modulo $\mathcal{H}$ in semigroups and rings, Linear Multilinear Algebra. 2013;61:886-891.

\bibitem{Patricio and Mendes} P. Patr\'{i}cio, C. Mendes Ara\'{u}jo, Moore-Penrose inverse in involutory rings: the case $aa^\dag=bb^\dag$, Linear Multilinear Algebra. 2010;58:445-452.

\bibitem{Patricio and Puystjens} P. Patr\'{i}cio, R. Puystjens, Drazin-Moore-Penrose invertibility in rings, Linear Algebra Appl. 2004;389:159-173.

\bibitem{Penrose} R. Penrose, A generalized inverse for matrices, Proc. Camb. Phil. Soc. 1955;51:406-413.

\bibitem{Zhu and chen} H.H. Zhu, J.L. Chen, P. Patr\'{i}cio, Further results on the inverse along an element in semigroups and rings, Linear Multilinear Algebra. 2016;64:393-403.

\bibitem{Zhu and p} H.H. Zhu, P. Patr\'{i}cio, J.L. Chen, Y.L. Zhang, The inverse along a product and its applications, Linear Multilinear Algebra. DOI: 10.1080/03081087.2015.1059796.

\bibitem{Zhu and zhang} H.H. Zhu, X.X. Zhang, J.L. Chen, Generalized inverses of a factorization in a ring with involution, Linear Algebra Appl. 2015;472:142-150.
\end{thebibliography}
\end{document}